\documentclass[11pt,leqno]{amsart}
\usepackage[latin1]{inputenc}
\usepackage{amssymb}
\usepackage{amsmath}
\usepackage{amsthm}
\usepackage{setspace} 
\usepackage{geometry}
\usepackage{graphicx}
\input{avv.sty}

\title{Hypergeometric functions and hyperbolic metric}
\author[Anderson, Sugawa, Vamanamurthy, and Vuorinen]%
{G. D. Anderson, T. Sugawa, M. K. Vamanamurthy, and M. Vuorinen}
\date{}


\newcounter{minutes}
\divide\time by 60
\newcounter{hours}
\multiply\time by 60
\addtocounter{minutes}{-\time}

\begin{document}

\newcommand{\ek}{{\mathcal{K}}}
\newcommand{\ee}{{\mathcal{E}}}
\newcommand{\K}{{\mathrm{K}}}
\newcommand{\E}{{\mathrm{E}}}
\newcommand{\uhp}{\mathbb{H}}
\newcommand{\D}{\mathbb{D}}
\renewcommand{\Re}{{\operatorname{Re}\,}}
\renewcommand{\Im}{{\operatorname{Im}\,}}
\newcommand{\RS}{{\widehat{\mathbb{C}}}} 
\newcommand{\PSL}{{\operatorname{PSL}}}
\newcommand{\arctanh}{{\operatorname{arctanh}\,}}
\newcommand{\inv}{^{-1}}


\begin{abstract}
We obtain new inequalities for certain hypergeometric functions.  Using these inequalities, we deduce estimates for the hyperbolic metric and the induced distance function on a certain canonical hyperbolic plane domain.
\end{abstract}

\maketitle
\markboth{\textsc{ANDERSON, SUGAWA, VAMANAMURTHY, AND
VUORINEN}}{\textsc{SPECIAL FUNCTIONS AND HYPERBOLIC METRIC}}

\section{Introduction}\label{sect:intro}

\newcommand{\sphere}{{\widehat{\mathbb{C}}}}

The hyperbolic metric is one of the most important tools for the
study of properties of analytic functions.
A plane domain $\Omega$ is called {\it hyperbolic} if it admits
a complete Riemannian metric of constant curvature $-4.$
The metric is called the {\it hyperbolic metric} of $\Omega$ and denoted by
$\rho_\Omega(z).$
We denote by $d_\Omega(z,w)$ the hyperbolic distance on $\Omega,$
which is induced by $\rho_\Omega.$
It is well known that a domain in the complex plane $\C$ is 
hyperbolic if and only if the boundary contains at least two points.
For instance, the unit disk $\D=\{z\in\C:|z|<1\}$ has
$\rho_\D(z)=1/(1-|z|^2)$ and $d_\D(z,w)=\arctanh (|z-w|/|1-\bar w z|).$

We recall the {\it principle of hyperbolic metric}, 
which is very useful in the study of conformal invariants \cite{Ah}.
Let $D$ and $\Omega$ be hyperbolic domains.
For an analytic map $f:D\to\Omega,$
this principle implies that the inequality
$\rho_\Omega(f(z))|f'(z)|\le\rho_D(z)$ holds, implying that
$d_\Omega(f(z),f(w))\le d_D(z,w)$
for $z,w\in D.$
Thus, a lower estimate for $d_\Omega(w_0,w)$ in terms of $|w|$
(for a fixed $w_0$) will lead to a growth estimate for an analytic
function $f: D\to \Omega$ with $f(z_0)=w_0$ for a fixed $z_0\in D.$
Similarly, a lower estimate for $\rho_\Omega$ yields a distortion
theorem for an analytic function on $D$ taking values in $\Omega.$

Since the twice-punctured plane $\C\setminus\{a,b\}$ is a
maximal hyperbolic domain, it is of particular importance.
We write $\lambda_{a,b}=\rho_{\C\setminus\{a,b\}}$ and
$d_{a,b}=d_{\C\setminus\{a,b\}}$ for short.
Noting the obvious relations
$$
\lambda_{0,1}(z)=|b-a|\lambda_{a,b}((b-a)z+a),
\quad
d_{0,1}(z,w)=d_{a,b}((b-a)z+a,(b-a)w+a),
$$
we may restrict our attention to the case when $a=0$ and $b=1.$
Precise information about $\lambda_{0,1}(z)$ and $d_{0,1}(z,w)$
leads to sharp forms of classical results in function theory
such as the Landau, Schottky, and Picard theorems
(cf.~\cite{Ah}, \cite{Hempel79}, \cite{Hempel80})
as well as various useful estimates for the hyperbolic metric of
a general plane domain (cf.~\cite{BP78}, \cite{SV}).


Since the inequalities
$$
\lambda_{0,1}(-|z|)\le \lambda_{0,1}(z)
\quad\text{and}\quad
d_{0,1}(-|z|, -|w|)\le d_{0,1}(z,w)
$$
hold for $z,w\in\C\setminus\{0,1\}$  (see \cite{LVV}),
lower estimates for $\lambda_{0,1}$ and $d_{0,1}$ reduce to
the analysis of these quantities on the negative real axis.

It is known (cf.~\cite{SV}) that $\lambda_{0,1}(-x)$ and $d_{0,1}(-x,-y)$
can be expressed in terms of complete elliptic integrals of the first kind
for $x,y>0$:
\begin{align*}
\lambda_{0,1}(-x)&=\frac{\pi}{8x\ek(r)\ek(r')}, \quad
\text{and} \\
d_{0,1}(-x,-y)&=|\Phi(x)-\Phi(y)|,
\end{align*}
where
\begin{align}
\label{eq:Phi}
\Phi(x)&=\frac12\log\frac{\ek(r)}{\ek(r')}, \\
\notag
r=\sqrt{\frac{x}{1+x}}, \quad
r'&=\sqrt{1-r^2}=\sqrt{\frac{1}{1+x}},
\end{align}
and
\begin{equation*}
\ek(r)=\int_0^1\frac{dt}{\sqrt{(1-t^2)(1-r^2t^2)}}
\end{equation*}
is the complete elliptic integral of the first kind.  
Note that the function $\Phi$ satisfies the relation $\Phi(1/x)=-\Phi(x)$
for $x>0.$
In particular, $\Phi(1)=0$ and hence $\Phi(x)=d_{0,1}(-x,-1)$ for $x\ge1.$
Note also that $\Phi$ can be expressed by
$2\Phi(x)=-\log(2\mu(r)/\pi),$
where $\mu(r)$ denotes the well-known modulus of the Gr\"otzsch ring
$\D\setminus[0,r]$ given by $\mu (r) =(\pi /2)\mathcal K(r')/\mathcal K(r)$,
for $r\in (0,1)$, and where $\D$ is the unit disk $\{z: |z|<1\}$ 
in the complex plane $\C$ (see \cite{LV} or \cite{AVV1} for details).

Since
\begin{equation}\label{eq:K}
\ek(r)=\frac\pi2 F\left(\frac{1}{2}, \frac{1}{2};1;r^2\right)
\end{equation}
in terms of the Gaussian hypergeometric function with specific parameters
(cf.~\cite[(3.13)]{AVV1}),
inequalities for hypergeometric functions
will lead to estimates for the hyperbolic metric.

In the present paper, we give several inequalities for hypergeometric
functions with restricted parameters.
By using these inequalities, we deduce estimates for the hyperbolic metric and
hyperbolic distance of the twice-punctured plane $\C\setminus\{0,1\}.$
In particular, we prove some of the conjectures proposed in \cite{SV}.
Inequalities for hypergeometric functions have
applications also in the study of the hyperbolic metric with conical
singularities as in [ASVV].

\section{Preliminaries}
In this section, we recall some definitions and properties of special functions
\cite{R}, and state two results that are particularly useful 
in proving monotonicity of a quotient of two functions \cite{AVV1},\cite{HVV}.

Given complex numbers
$a$, $b$, and $c$ with $c\neq 0,-1,-2, \dots, $
the \emph{Gaussian hypergeometric function} is defined as

\begin{equation}\label{eq:kolmas}
F(a,b;c;z)\!= \!{}_2 F_1(a,b;c;z)\! \equiv\!
\sum_{n=0}^{\infty} \frac{(a,n)(b,n)}{(c,n)} \frac{z^n}{n!},
\quad |z|<1,
\end{equation}

\noindent 
and then is continued analytically to the slit plane $\C\setminus [1,\infty)$.
Here $(a,n)$
is the \emph{shifted factorial function}, namely, $(a,0)=1$ and
\begin{equation*}
(a,n)\equiv a(a+1)(a+2) \cdots (a+n-1)
\end{equation*}
for $n=1,2,3,\ldots $.  For our applications to hypergeometric functions, 
in the sequel we will need only real parameters $a,b,c$
and real argument $z=x$.

It is well known that the hypergeometric function $v=F(a,b;c;x)$
satisfies the hypergeometric differential equation
\begin{equation}\label{eq:DE}
x(1-x)v''+[c-(a+b+1)x]v'-ab v=0.
\end{equation}

We recall the derivative formula
\begin{equation}\label{eq:diff}
F'(a,b;c;x)=\frac{d}{dx}F(a,b;c;x)=\frac{ab}{c}F(a+1,b+1;c+1;x).
\end{equation}

The behavior of the hypergeometric function near $x=1$ is given by

\begin{equation}\label{eq:viides}
\begin{cases}
F(a,b;c;1) = \dfrac{\Gamma(c) \Gamma(c-a-b)}{\Gamma(c-a) \Gamma(c-b)},~
c>a+b,\\
F(a,b;a+b;1-) = \dfrac1{B(a,b)}\Big[\log\dfrac{1}{1-x}+R(a,b)\Big](1+O(1-x)),\\
F(a,b;c;x) = (1-x)^{c -a -b} F(c-a,c-b;c;x),~ c<a+b.
\end{cases}
\end{equation}
Here  $B(a,b)$ denotes the beta function, namely,
$$
B(a,b)=\frac{\Gamma(a)\Gamma(b)}{\Gamma(a+b)},
$$
and $R(a,b)$ is the function defined by
$$
R(a,b) = -2 \gamma - \Psi(a) - \Psi(b),
$$
where $\Psi(x)=\Gamma'(x)/\Gamma(x)$ is the digamma function and
$\gamma$ is the Euler-Mascheroni constant, known to equal $-\Psi(1).$
In particular, $B(1/2,1/2) = \pi$ and $R(1/2,1/2) = \log 16$.
The asymptotic formula  in (\ref{eq:viides}) for the \emph{zero-balanced} case $a+b=c$
is due to Ramanujan.
A general formula appears in \cite[15.3.10]{AS}.
Some explicit estimates for the $O(1-x)$ term in the zero-balanced case
are given in \cite[Theorem 1.52]{AVV1}.

\bigskip

The following lemmas will be used in the next section.

\begin{lem}\label{lem:diff}
For $a,b$ with $a+b\ne 0,-1,-2,\dots,$
$$
(1-x)\frac{d}{dx}F(a,b;a+b;x)=\frac{ab}{a+b} F(a,b;a+b+1;x).
$$
\end{lem}

\proof
In view of \eqref{eq:diff}, it is enough to show the identity
$F(a+1,b+1;a+b+1;x)=(1-x)^{-1}F(a,b;a+b+1;x),$ which follows
from the third case of \eqref{eq:viides}.
\qed

\bigskip

A sequence $\{a_n\}_{n=0}^\infty$ of real numbers
is called {\it totally monotone} (or {\it completely monotone}) if
$\Delta^k a_n\ge0$ for all $k,n=0,1,2,\dots.$
Here, $\Delta^k a_n$ is defined inductively in $k$ by
$\Delta^0 a_n=a_n$ and $\Delta^{k+1}a_n=\Delta^ka_n-\Delta^ka_{n+1}.$

\begin{rem}\label{rem:Dirac}
For a totally monotone sequence $\{a_n\},$
$\Delta^ka_n>0$ for $k,n=0,1,2,\dots$
unless $a_1=a_2=a_3=\dots.$

Indeed, Hausdorff's theorem (cf.~\cite{Wall:fraction}) 
tells us that $\{a_n\}$ is
totally monotone precisely when there exists a positive Borel measure
$\nu$ on $I=[0,1]$ such that
$$
a_n=\int_I x^n d\nu(x), \quad (n=0,1,2,\dots).
$$
Since
$$
\Delta^ka_n=\int_I x^n(1-x)^k d\nu(x),
$$
we have the strict inequality $\Delta^ka_n>0$
unless $\nu$ is a linear combination of the Dirac measures
$\delta_0$ at $x=0$ and $\delta_1$ at $x=1.$
When $\nu=a\delta_0+b\delta_1$ with $a,b\ge0,$ we have
$a_0=a+b$ and $a_1=a_2=\dots=b.$
\end{rem}

K\"ustner \cite{Kus02} studied hypergeometric functions from the aspect
of totally monotone sequences.
We need the following result for later use.

\begin{lem}[$\text{K\"ustner \cite[Theorem 1.5]{Kus02}}$]
\label{lem:Kustner}
Let $a,b,c$ be real numbers with $-1\le a\le c$ and $0<b\le c.$
The coefficients of the power series expansion
of the function $F(a+1,b+1;c+1;x)/F(a,b;c;x)$ about $x=0$
form a totally monotone sequence.
\end{lem}

\bigskip

The following general lemma will be a useful tool for 
proving properties of hypergeometric functions in the next section.
We remark that this is a special case of a more general result
\cite[Lemma 2.1]{PV} (see also \cite[Theorem 4.3]{HVV}).
\bigskip

\begin{lem}\label{lem:new}
Let $f(x)=\sum_{n=0}^{\infty } a_nx^n$ be a real power series convergent 
on $(-1,1)$, and let the sequence $\{a_n\}$ be non-constant and monotone.  
Let $g(x)=(1-x)f(x)=\sum_{n=0}^{\infty} b_nx^n$.
\begin{enumerate}
\item[(1)] If the sequence $\{a_n\}$ is non-decreasing, 
then $b_n \ge 0$, for $n=1, 2, \ldots,$ with strict inequality 
for at least one $n$.  In particular, $g$ is strictly increasing on $[0,1)$.
\item[(2)] If the sequence $\{a_n\}$ is non-increasing, 
then $b_n\le 0$, for $n=1,2,\ldots$, with strict inequality 
for at least one $n$.  In particular, $g$ is strictly decreasing on $[0,1)$.
\end{enumerate}
\end{lem}
\begin{proof}
Clearly, $b_0=a_0$ and $b_n=a_n-a_{n-1}$, for $n=1,2,3,\ldots,$ 
so that the assertions follow immediately.
\end{proof}
\vspace{.1in}

\begin{rem}\label{rem:new}
In Lemma \ref{lem:new}, if the sequence $\{a_n\}$ is constant, 
then clearly $g(x)$ is constant and equals $a_0$ for all $x$.
\end{rem}
\bigskip

\section{Some properties of hypergeometric functions}

In the present section, we investigate properties of some combinations
of hypergeometric functions.
Some of these will be applied to hyperbolic metric in the next section.

\begin{lem}\label{lem:Vaman}
For positive numbers $a,b,c$ and $x \in (-1,1)$, we let 
$f(x) = (1-x)F(a,b;c;x)$.
\begin{enumerate}
\item[(1)] If $ab \ge c$, and $a+b \ge c+1$, with at least one of these
inequalities being strict, then $f'(x)>0$ on $(0,1)$, so that $f$ is strictly increasing.
\item[(2)] If $ab \le c$, and $a+b \le c+1$, with at least one of these
inequalities being strict, then $f'(x)<0$ on $(0,1)$, so that $f$ is strictly decreasing.
\item[(3)] If both of the inequalities become equalities, then $f$ is the constant function $1$.
\end{enumerate}
\end{lem}
\begin{proof}
$F(a,b;c;x) = \sum_{n=0}^{\infty} T_nx^n$, where $T_n= (a,n)(b,n)/[(c,n)n!]$.  Hence, $T_{n+1}/T_n = (a+n)(b+n)/[(c+n)(n+1)]$, for $n=0,1,2,\ldots$, which is non-decreasing in case (1), non-increasing in case (2), and constant in case (3).  Hence the assertions follow from Lemma \ref{lem:new} and Remark \ref{rem:new}.
\end{proof}
\bigskip
\begin{lem}\label{lem:concave}
Let $a,b,c$ be positive numbers with $\max\{a,b\}<c$
and set $v(x)=F(a,b;c;x)$ for $x\in(0,1).$
Then the function
$$
f(x)=x(1-x)\frac{v'(x)}{v(x)}
$$
is positive and has negative Maclaurin coefficients except for the linear term.
In particular, $f(x)$ is strictly concave on $(0,1)$.
\end{lem}

\proof
First, the positivity of $f(x)$ is obvious from (\ref{eq:diff}). Next, we expand $v'(x)/v(x)$ as a power series
$a_0+a_1x+a_2x^2+\dots$ in $|x|<1.$
Lemma \ref{lem:Kustner}, together with \eqref{eq:diff}, implies
that $\{a_n\}$ is a totally monotone sequence and, in particular,
that $0\le a_{n}\le a_{n-1}$ for $n=1,2,3,\dots.$
If one of the inequalities were not strict for some $n,$
by Remark \ref{rem:Dirac} we would have
$v'(x)/v(x)= \alpha/(1-x)+\beta$
for some constants $\alpha, \beta\ge0.$
Then $v(x)=e^{\beta x}(1-x)^{-\alpha}.$
We substitute $v'=(\alpha/(1-x)+\beta)v$ and
$v''=[(\alpha/(1-x)+\beta)^2+\alpha/(1-x)^2]v$ into
\eqref{eq:DE} to get the relation
$$
x[(\alpha+\beta(1-x))^2+\alpha]+(\alpha+\beta(x-1))[c-(a+b+1)x]
-ab(1-x)=0.
$$
Equating the coefficients to zero, we get $\beta = 0$, 
$\alpha = ab/c$, and $ab+c^2 = (a+b)c$, 
so that $(c-a)(c-b)=0$, a contradiction.

Therefore, we have $0<a_n<a_{n-1}$ for $n=1,2,3,\dots.$
Since
$$
f(x)=a_0x+\sum_{n=1}^\infty(a_n-a_{n-1})x^{n+1},
$$
the first assertion follows.
We also have $f''(x)<0,$ from which the strict concavity follows.
\qed
\vspace{.1in}

\begin{rem}\label{rem:simpler}
When $\max\{a,b\}=c,$ we have an even simpler conclusion.
Assume, for instance, that $0<a\le b=c.$
Then $v(x)=(1-x)^{-a}$ and thus $f(x)=ax.$
\end{rem}
\vspace{.1in}

\begin{cor}\label{cor:concave}
Let $a,b,c$ be positive numbers with $\max\{a,b\}\le c$
and set $v(x)=F(a,b;c;x)$ and
$$
N(x)=x(1-x)\left[\frac{v'(x)}{v(x)}+\frac{v'(1-x)}{v(1-x)}\right]
$$
for $x\in(0,1).$
When $\max\{a,b\}<c,$ the function $N(x)$
is positive, symmetric about the point $x=1/2,$ strictly concave on $(0,1),$
strictly increasing on $(0,1/2]$, and
strictly decreasing on $[1/2,1).$
When $\max\{a,b\}=c,$ $N(x)$ is a positive constant.
\end{cor}

\proof
Since $N(x)=f(x)+f(1-x),$ where $f$ is as in Lemma \ref{lem:concave}, the assertions follow from Lemma \ref{lem:concave} and Remark \ref{rem:simpler}.
\qed
\bigskip

We remark that the above function can be written in the form
$N(x)=M(x)/[v(x)v(1-x)],$ where $M$ is the Legendre $M$-function of the hypergeometric function $v$ given by
$$
M(x)=x(1-x)\left[ v'(x)v(1-x)+v(x)v'(1-x)\right],\quad 0<x<1.
$$
See \cite{HVV} and \cite{HLVV} for details.
Obviously, $M(x)$ is symmetric about the point $x=1/2,$
that is, $M(1-x)=M(x)$ for $x\in(0,1).$
We will need the following property of $M$ (see \cite{HLVV}).

\begin{lem}\label{lem:HLVV}
Let $a,b,c$ be positive numbers, let $v(x) = F(a,b;c;x)$, and let $M(x)$ be its Legendre $M$-function.  Then  $M$ is strictly convex, strictly concave, or constant, according as $(a+b-1)(c-b)$ is positive, negative, or zero.
\end{lem}
\bigskip

\begin{thm}\label{thm:genconv}
 For positive $a,b,c$ and $x \in (-1,1),$ let $F(x) =
F(a,b;c;x),$ and let $f(x) = F(x)F(1-x).$ If 
$ab/(a+b+1) < c,$ then $f'/f$
is strictly increasing on $(0,1),$ vanishing at $x = 1/2.$ 
Hence, $f$ is strictly log-convex on $(0,1),$ strictly decreasing on 
$(0,1/2),$
strictly increasing on $(1/2, 1),$ with minimum value at $ x = 1/2.$
Further,
\begin{enumerate}
 \item[(1)] If $a+b < c,$ then $f(1-) = F(1) =
\Gamma(c) \Gamma(c-a-b) / [\Gamma(c-a) \Gamma(c-b)],$ 
\item[(2)] If  $ a+b = c,$ then $$f(x) =
(1/B(a,b))[-\log(1-x) + R(a,b)] + O[(1-x)\log(1-x)],$$ as $x \to  1-,$ 
\item[(3)] If $a+b > c,$ then 
$$f(x) = [\Gamma(c) \Gamma(a+b-c)/(\Gamma(a)\Gamma(b))](1-x)^{c-a-b} +o(1),$$ 
as $x \to  1- $.
\end{enumerate}
\end{thm}
\begin{proof}
By logarithmic differentiation, 
$$\frac{f'(x)}{f(x)} = \frac{F'(x)}{F(x)} -\frac{F'(1-x)}{F(1-x)},$$
so that the assertions of convexity and monotonicity
follow from \cite[Theorem 1.3.(1)]{AVV2}. The asymptotic
relations follow from (\ref{eq:diff}). 
\end{proof}
\vspace{.1in}

The next result states some properties of the product 
$$
F(a,b;a+b;x)F(a,b;a+b;1-x)\ \ {\rm with}\ \ x=e^t/(1+e^t),
$$  
which reduces essentially to the function $H(t)$
given in Theorem \ref{thm:C2.12} below when $a=b=1/2.$
\vspace{.1in}

\begin{thm}\label{thm:main}
Let $a$ and $b$ be positive numbers with $ab<a+b.$
Define a function $P$ on $\R$ by
$$
P(t) = F\left(a,b;a+b;\frac{e^t}{1+e^t}\right)F\left(a,b;a+b;\frac{1}{1+e^t}\right).
$$
Then,
\begin{enumerate}
\item[(1)] $P$ is even, strictly decreasing on $(-\infty ,0]$, 
and strictly increasing on $[0,\infty )$, and $P''(t)>0$ 
(so that $P$ is strictly convex on $\R$).  Moreover,
$$
P(t) = [|t|+R(a,b)]/B(a,b) + O(te^{-|t|}),\ \ {\rm as}\ \ t \to +\infty\ \ {\rm or}\  -\infty.
$$
\item[(2)] The derivative $P'$ is odd, strictly increasing on $\R$,
such that $P'(0) = 0$, and $P'(t) = (|t|/t)B(a,b) + O(te^{-|t|})$, as $t \to +\infty$ or $-\infty$. In particular, we have the sharp inequalities
$- 1/B(a,b) < P'(t) < 1/B(a,b)$, for all $t \in \R$.
\item[(3)] The function $P(t)-t/B(a,b)$ is strictly convex and decreasing, while $P(t) + t/B(a,b)$ is strictly convex and increasing on $\R$. In particular, we have the sharp inequality
$R(a,b)/B(a,b) < P(t) - |t|/B(a,b) \le P(0)$, for all $t \in \R$.
\item[(4)] The function $G(t)=(P(t)-P(0))/t$ is strictly increasing from $\R$
onto $(-1/B(a,b),1/B(a,b))$.
\end{enumerate}
\end{thm}
\begin{proof}
The assertions that $P$ is even, hence $P'$ is odd, are obvious.  By symmetry, it is enough to prove the assertions only on $(0,\infty )$.
We next show the rest of (1) and (2).
For brevity, we write $c=a+b, A=1/B(a,b), g(t)=e^t/(1+e^t),$
$v(x)=F(a,b;c;x), v_1(x)=v(1-x), f(x)=v(x)v_1(x),$ and
$w(x)=F(a,b;c+1;x).$
Then, by Lemma \ref{lem:diff}, we obtain
$(1-x)v'(x)=abw(x)/c$ and $xv_1'(x)=-abw(1-x)/c.$
We now put $x=g(t).$
Then $P(t)=f(g(t)), g'(t)=x(1-x)$, and thus
\begin{align*}
P'(t) & = x(1-x)f'(x) = x(1-x)[v'(x)v_1(x)+v(x)v_1'(x)] \\
& = \frac{ab}{c}[xw(x)v(1-x)-(1-x)w(1-x)v(x)] \\
&= \frac{ab}{c}[L(x)-L(1-x)],
\end{align*}
where $L(x)=xv(1-x)w(x).$
Since $w(x)>0$ and $w'(x)>0$ and since
$(d/dx)[xv(1-x)]>0$ on $(0,1)$
by Lemma \ref{lem:Vaman}, we have $L'(x)>0$.
Therefore, $$
P''(t)=\frac{ab}{c}x(1-x)[L'(x)+L'(1-x)]>0,\quad t>0,
$$
and hence $P$ is strictly convex.

In order to observe the asymptotic behavior of $P$ and $P',$
we note that $x=g(t)$ satisfies the relation 
$-\log(1-x)=\log(1+e^t)=t+O(e^{-t})$ as $t\to+\infty.$
By \eqref{eq:viides}, we see that
\begin{align*}
 P(t)=v(x)v(1-x)=&A[-\log(1-x)+R(a,b)](1+O(1-x))\\
 =&A[t+R(a,b)](1+O(e^{-t})),\\
\end{align*}
as $t\to+\infty.$ Thus the proof of (1) is now complete.

For $P',$ we need to study the behavior of $w(x).$
By \eqref{eq:viides}, $w(x)\to
cA/(ab)$ as $x\to1-.$
More precisely, by the asymptotic expansion in \cite[15.3.11]{AS},
we obtain
$$
w(x)=\frac{\Gamma(c+1)}{\Gamma(a+1)\Gamma(b+1)}+O(-(1-x)\log(1-x))
$$
as $x\to1-.$
Therefore, $L(x)=cA/(ab)+O(-(1-x)\log(1-x)).$
Since $L(1-x)=O(-(1-x)\log(1-x))$ by \eqref{eq:viides},
we now have $P'(t)=A+O(te^{-t})$ as required.

Assertion (3) follows immediately from (2).
Since (2) implies that $P$ is strictly convex, 
the slope $G(t)$ is strictly increasing.
\end{proof}

\bigskip

The next result gives properties of the quotient of
$F(a,b;a+b;x)$ over $F(a,b;a+b;1-x)$ with $x=e^t/(1+e^t).$
\bigskip

\begin{thm}\label{thm:main2}
Let $a$ and $b$ be positive numbers.
Define functions $Q$ and $q$ on $\R$, respectively, by
$$
Q(t) = \frac{F\left(a,b;a+b;\frac{e^t}{1+e^t}\right)}{F\left(a,b;a+b;\frac{1}{1+e^t}\right)}
\quad\text{and}
\quad
q(t)=\log Q(t).
$$
Then, the following hold:
\begin{enumerate}
\item[(1)] $Q$ is a strictly increasing positive function on $\R$ with the properties
$Q(t)Q(-t)=1$ and $Q(t)=B(a,b)^{-1}[t+R(a,b)]+O(t e^{-t})$ as $t\to+\infty,$
\item[(2)] $q$ is a strictly increasing odd function on $\R$
satisfying $q(t)=\log t-\log B(a,b)+O(1/t)$ as $t\to+\infty,$
\item[(3)] $q'$ is strictly increasing on $(-\infty, 0]$ and strictly decreasing on $[0,\infty )$, so that $q$ is strictly convex on $(-\infty ,0)$ and
strictly concave on $(0,\infty).$
\item[(4)] $q(t)/t$ is strictly decreasing on $(0,\infty )$ and hence,
$q$ is subadditive on $(0,+\infty),$ that is,
$q(t+t')\le q(t)+q(t')$ for $t,t'>0.$
\item[(5)] $Q(t)-t/B(a,b)$ is strictly decreasing and
strictly convex on $(0,\infty )$ when $a+b\ge1.$
\item[(6)] $(R(a,b)+t)/B(a,b)<Q(t)<1+t/B(a,b)$ on $(0,\infty )$ when $a+b\ge1.$
\end{enumerate}
\end{thm}

\proof
We put $g(t)=e^t/(1+e^t)$ and $v(x)=F(a,b;a+b;x)$ as in the proof
of Theorem \ref{thm:main}.
We further set $k(x)=\log[v(x)/v(1-x)]$ for $x\in(0,1).$
Then $q=k\circ g$ and thus
 $$
 q'(t)=k'(g(t))g'(t)=x(1-x)[v'(x)/v(x)+v'(1-x)/v(1-x)]=N(x),
 $$ 
where $x=g(t)$ and $N(x)$ is given in Corollary \ref{cor:concave}.
Since $N(x)>0$ by the corollary, we conclude that $q$ and $Q$ are
both strictly increasing. Positivity and the relation 
$Q(t)Q(-t)=1$ immediately follow from the definition.
The asymptotic behavior of $Q$ follows from the relation
$P(t)=Q(t)v(1-x)^2,\, x=g(t),$ and Theorem \ref{thm:main}(1).
Taking the logarithm in (1), we also obtain the asymptotic behavior of $q$
asserted in (2).

We next show (3).
Since $q$ is odd, it suffices to show that
$q'$ is strictly decreasing on $(0,\infty).$
As we saw above, $q'(t)=N(g(t)).$
By Corollary \ref{cor:concave}, $N(x)$ is strictly decreasing
in $1/2<x<1.$
Therefore, $q'(t)$ is strictly decreasing on $(0,\infty ),$
so that $q(t)$ is strictly concave on $(0,\infty )$.

For (4), since $q$ is strictly concave by (3), it follows that the slope $q(t)/t$ is strictly decreasing on $(0,\infty )$.  The subadditivity follows from \cite[Lemma 1.24]{AVV1}.

To show (5), we put $f(t)=Q(t)-t/B(a,b).$
Then, by $Q'(t)=q'(t)Q(t)=N(x)v(x)/v(1-x)$ with $x=e^t/(1+e^t),$
we have
$$
f'(t)=\frac{N(x)v(x)}{v(1-x)}-\frac1{B(a,b)}
=\frac{M(x)}{[v(1-x)]^2}-\frac1{B(a,b)}.
$$
By Lemma \ref{lem:HLVV}, when $a+b\ge1,~M(x)$ is increasing 
and positive in $1/2<x<1$
and $v(1-x)^2$ is strictly decreasing on $1/2<x<1.$
Thus $f'(t)$ is strictly increasing on $1/2<x<1,$ and hence on $0<t<\infty ,$
which means that $f(t)$ is strictly convex in $0<t<\infty .$
Now $M(x)=ab[L(x)+L(1-x)]/(a+b),$ where $L(x)$ is as
in the proof of Theorem \ref{thm:main}.
Thus, by (\ref{eq:viides}), we see that $M(x)\to 1/B(a,b)$ as $x\to1-.$
Since $f'(t)$ is strictly increasing in $1/2<x<1,$
we conclude that $f'(t)<\lim_{t\to+\infty}f'(t)=0,$ which
implies that $f(t)$ is strictly decreasing.

Finally, (6) follows from (5) and the fact that $f(0)=1$
and $\lim_{t\to+\infty}f(t)=R(a,b)/B(a,b)$
by \eqref{eq:viides}.
\qed
\bigskip



\bigskip

\section{Applications to hyperbolic metric}

In \cite{SV}, the function
\begin{equation}\label{eq:h}
h(t)=e^t\lambda_{0,1}(-e^t)=\pi\left[
8\ek\Big(\frac1{\sqrt{1+e^t}}\Big)\ek\Big(\frac1{\sqrt{1+e^{-t}}}\Big)
\right]^{-1},\  t\in \R,
\end{equation}
plays a special role in the estimation of the hyperbolic metric of a general
plane hyperbolic domain.
Let $\Omega$ be a hyperbolic domain in $\C$ and set
$$
m(a,s)=\inf_{b\in\partial\Omega}\big|s-\log|b-a|\big|
$$
for $a\in\partial\Omega$ and $s\in\R.$
Then we have, for instance, 
$$
h(m(a,\log|z-a|))\le |z-a|\rho_\Omega(z)\le \frac\pi{4m(a,\log|z-a|)},
\quad z\in\Omega,
$$
for every $a\in\partial\Omega.$
By the inequality $h(t)\ge 1/(2|t|+2C_0),$ which is essentially due to J.~Hempel,
we can reproduce the sharp version of the Beardon-Pommerenke inequality
(see \cite{SV} for details). Here,
\begin{equation}\label{eq:Hart}
C_0=\frac1{2\lambda_{0,1}(-1)}=\frac4\pi\ek(1/\sqrt2)^2
=\frac{\Gamma(1/4)^4}{4\pi^2}\approx 4.37688.
\end{equation}

In the course of their investigation, the second and fourth authors
arrived at some conjectures (Conjecture 2.12 in \cite{SV}), 
which we are now able to prove.
We remark that there is a slight error in the statement of Conjecture 2.12(3)
of \cite{SV}: the interval $(-\pi/4,\pi/4)$ has to be replaced by
$(-2,2)$ as in (3) below.

\begin{thm}\label{thm:C2.12}
\hfill

\begin{enumerate}
\item[(1)] The function $t\,h(t)$ is strictly increasing from $(0,\infty )$ onto $(0,1/2)$.
\item[(2)] The even function $H(t)=1/h(t)$ satisfies the condition $H''(t)>0$ and is a strictly convex self-homeomorphism of $\R$.
\item[(3)] The odd function $H'(t)$ maps $\R$ homeomorphically
onto the interval $(-2,2).$
\item[(4)] $2(|t|+C_0)h(t)<1.25$ for $t\in\R.$
\item[(5)] $1/(|t|+C_0)<2h(t)<1/(|t|+\log 16)$, for $t\in \R$.
\end{enumerate}
\end{thm}

\proof
First, in view of \eqref{eq:K}, we have the relation
\begin{equation}
H(t)=\frac1{h(t)}=2\pi P(t),
\end{equation}
where $P(t)$ is the function defined in Theorem \ref{thm:main}
with the parameters $a=b=1/2$,
so that $B(1/2,1/2)=\pi.$
Thus assertions (2) and (3) follow, respectively, from (1) and (2) of Theorem \ref{thm:main}.

Assertion (1) is equivalent to the statement that $H(t)/t$ is strictly decreasing on $(0,\infty )$,
which follows from Theorem \ref{thm:main}(3) since $R(1/2,1/2) = \log 16 > 0$.  The limiting values are clear.

For (4),
since $h$ is even, it suffices to show the inequality for $t\ge 0.$
Let
$$
G(t)=(t+C_0)h(t)=\frac{t+C_0}{H(t)}.
$$
Then $G'(t)=g(t)/H(t)^2,$ where $g(t)=H(t)-(t+C_0)H'(t)=2\pi[P(t)-(t+C_0)P'(t)].$
Since $g'(t)=-(t+C_0)H''(t)<0$ by (2), the function $g$ is
strictly decreasing on $[0,+\infty).$
Noting that $H'(0)=0,$ we see that $g(0)=H(0)>0.$
By Theorem \ref{thm:main} (1) and (2),
$$
P(t)-(t+C_0)P'(t)=\frac{R(\tfrac12,\tfrac12)-C_0}{B(\tfrac12,\tfrac12)}+O(t^2e^{-t})
$$
as $t\to+\infty.$
Since
$$
R(\tfrac12,\tfrac12)-C_0=2\log4-\frac{\Gamma(1/4)^4}{4\pi^2}
\approx -1.6043,
$$
we see that $\lim_{t\to+\infty}g(t)<0.$
Therefore, there is a unique zero $t=t_0$ of $g(t)$ on $(0,\infty )$, so that $G$ is strictly increasing on $[0,t_0]$ and strictly decreasing on $[t_0,\infty)$.  Thus, the function $G(t)$ takes its maximum at $t=t_0.$
Hence,
$$
\max_{0\le t<\infty } G(t)=\frac{t_0+C_0}{H(t_0)}=\frac1{H'(t_0)}.
$$
Let $t_1=2.56.$
Then, by a numerical computation, we observe that
$g(t_1)>0.02$ and thus $t_1<t_0.$
Since $H'(t)=2\pi P'(t)$ is increasing by Theorem \ref{thm:main}(2),
by another numerical computation we conclude that
$$
\max_{t\in\R}2(|t|+C_0)h(t)=\frac2{H'(t_0)}<\frac2{H'(t_1)}<1.248<1.25.
$$

Finally, (5) follows from Theorem \ref{thm:main}(3), if we put $a = b = 1/2$ and observe that $R(1/2,1/2) = \log 16$ and $B(1/2,1/2) = \pi$.
\qed
\bigskip

A numerical experiment suggests that $t_0\approx 2.56944$
and $2/H'(t_0)\approx 1.24477.$

Theorem \ref{thm:C2.12} has an application to the hyperbolic metric.
For a hyperbolic domain $\Omega$ in $\C,$ we consider the quantity
$$
\sigma_\Omega(z)=\sup_{a,b\in\partial\Omega}\lambda_{a,b}(z),\quad
z\in\Omega.
$$
Since $\Omega\subset\C\setminus\{a,b\},$
we have $\rho_\Omega(z)\ge\lambda_{a,b}(z)$ for $a,b\in\partial\Omega.$
Thus, $\sigma_\Omega(z)\le\rho_\Omega(z).$
Gardiner and Lakic \cite{GL01} proved that $\rho_\Omega(z)\le
A \sigma_\Omega(z)$ for an absolute constant $A.$
We denote by $A_0$ the smallest possible constant $A.$
In \cite{SV} it is shown that
$A_0\le 2C_0+\pi/2\approx10.3246$
and observed (see Remark 3.2 in \cite{SV})
that this could be improved to $A_0\le 1/h(\pi/4)=H(\pi/4)
\approx 9.0157$ if assertion (1) in Theorem \ref{thm:C2.12} were true.
We now have this assertion.
We remark that Betsakos \cite{Bet} recently proved a stronger
inequality which leads to $A_0\le 8.27.$

In \cite{SV}, the function $\varphi(t)=2\Phi(e^{t/2})$
plays an important role in the estimation of the hyperbolic
distance, where $\Phi(x)$ is as in (\ref{eq:Phi}).
In view of \eqref{eq:Phi} and \eqref{eq:K},
we have the expression
$$
\varphi(t)=\log
\frac{F\big(\tfrac12,\tfrac12;1;\frac{e^{t/2}}{1+e^{t/2}}\big)}%
{F\big(\tfrac12,\tfrac12;1;\frac{1}{1+e^{t/2}}\big)}
=q(t/2),
$$
where $q$ is the function defined in Theorem \ref{thm:main2}
with $a=b=1/2.$
Thus, as a corollary of Theorem \ref{thm:main2}, we obtain the following.

\begin{cor}\label{cor:phi}
The function $\varphi(t)/t$ is strictly decreasing on $(0,\infty )$, implying that 
$\varphi(t)$ is subadditive on $(0,\infty )$.
\end{cor}

The statement of Corollary \ref{cor:phi} was given as
Conjecture 5.9 in \cite{SV}.
This conjecture has recently been settled by Baricz \cite{Bar07}
by a different method.

As an application of Corollary \ref{cor:phi}, we give an
improvement of Theorem 5.12 in \cite{SV}.

\begin{thm}\label{thm:appl}
Let $a_0, a_1, a_2, \dots$ be an infinite sequence of
distinct complex numbers with the properties
\begin{enumerate}
\item[(1)] $0=|a_0|<|a_1|\le |a_2|\le \dots,$
\item[(2)] $|a_{n+1}|\le e^c |a_n|$ for $n=1,2,3,\dots,$
where $c>0$  is a constant, and
\item[(3)] $|a_n|\to\infty$ as $n\to\infty.$
\end{enumerate}
If a domain $\Omega$ in $\C$ omits all the points $a_n,$ then
\begin{equation}\label{eq:le}
d_\Omega(z_1,z_2)\ge
A\big(\log|z_2|-\log|z_1|\big)-B
\end{equation}
for $z_1, z_2\in\Omega$ with $e^{-c/2}|a_1|\le |z_1|\le |z_2|.$
Here, $A=\varphi(c)/c$ and $B=\varphi(c)-\varphi(c/2).$
\end{thm}

\proof
We use the same argument indicated in \cite[p.~901]{SV}.

Choose integers $k$ and $l$ with $1\le k\le l$, so that
$\sqrt{|a_ka_{k-1}|}\le |z_1|\le\sqrt{|a_ka_{k+1}|}$ and
$\sqrt{|a_la_{l-1}|}\le |z_2|\le\sqrt{|a_la_{l+1}|}$ hold.
Then, Theorem 1.7 in \cite{SV} gives us
$$
d_\Omega(z_1,z_2)\ge\frac12\varphi(t_k-\log|z_1|)+
\sum_{n=k+1}^l \varphi(t_n-t_{n-1})
+\frac12\varphi(\log|z_2|-t_l),
$$
where $t_n=\log|a_n|.$
Noting that $t_n-t_{n-1}\le c,~t_k-\log|z_1|\le c/2$ and $\log|z_2|-t_l\le c/2,$
we deduce from the monotonicity of $\varphi(t)/t$ the chain of inequalities
\begin{align*}
&\quad d_\Omega(z_1,z_2) \\
&\ge \frac12 (t_k-\log|z_1|)\frac{\varphi(c/2)}{c/2}
+\sum_{n=k+1}^l (t_n-t_{n-1})\frac{\varphi(c)}{c}
+\frac12 (\log|z_2|-t_l)\frac{\varphi(c/2)}{c/2} \\
&=\frac{\varphi(c)}{c}(\log|z_2|-\log|z_1|)
-\frac{\varphi(c)-\varphi(c/2)}{c}(t_k-\log|z_1|+\log|z_2|-t_l) \\
&\ge\frac{\varphi(c)}{c}(\log|z_2|-\log|z_1|)+\varphi(c)-\varphi(c/2).
\end{align*}
\qed

Under the same hypothesis as in Theorem \ref{thm:appl},
Theorem 5.12 in \cite{SV} asserts \eqref{eq:le} with
$A=h(c/2)$ and $B=0,$ where $h$ is given by \eqref{eq:h}.
Also, (5.17) of \cite{SV} gives \eqref{eq:le} with
$A=(1/c)\log(1+c/(2C_0))$ and $B=c/(4\pi).$
Compare the graphs of the functions $\varphi(c)/c, h(c/2)$ and
$(1/c)\log(1+c/(2C_0))$ (Figure \ref{fig:1}).

\begin{figure}[htbp]
\includegraphics[width=10cm, height=7.5cm]{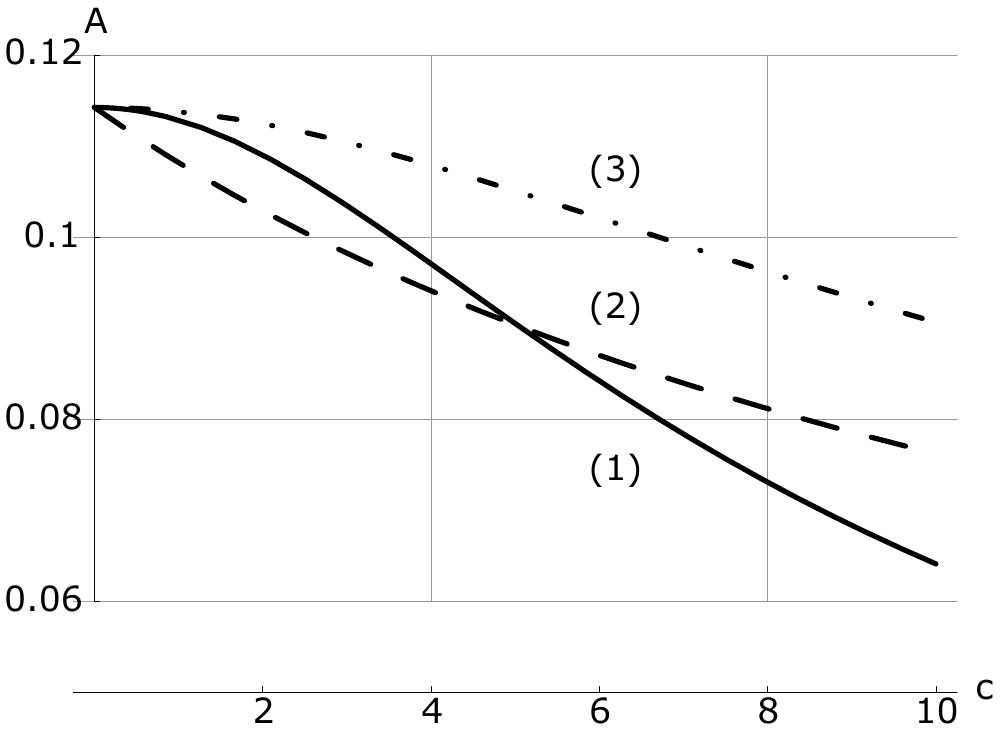}
\caption{Graphs of $\varphi(c)/c$ (solid line), $h(c/2)$ (dashed line) and $(1/c)\log(1+c/(2C_0))$
(dotted line)}
\label{fig:1}
\end{figure}

\bigskip\bigskip


\bigskip

\noindent
ANDERSON: \\
  Department of Mathematics \\
Michigan State University \\
     East Lansing, MI 48824, USA \\
      email: {\tt anderson@math.msu.edu}\\
     FAX: +1-517-432-1562\\ [1mm]

\noindent
SUGAWA:\\
     Department of Mathematics, Graduate School of Science \\
     Hiroshima University, 1-3-1, Kagamiyama\\
    Higashi-Hiroshima, 739-8526 JAPAN\\
     email: {\tt sugawa@math.sci.hiroshima-u.ac.jp}\\
     FAX: +81-82-424-0710\\ [1mm]

\noindent
VAMANAMURTHY:\\
   Department of Mathematics \\
    University of Auckland \\
    Auckland, NEW ZEALAND\\
     email: {\tt vamanamu@math.auckland.nz}\\
FAX: +649-373-7457\\ [2mm]

\noindent
VUORINEN:\\
     Department of Mathematics \\
     University of Turku \\
     Vesilinnantie 5\\
     FIN-20014, FINLAND\\
     e-mail: ~~{\tt vuorinen@utu.fi}\\
     FAX: +358-2-3336595\\

\end{document}